\documentclass[A4paper,12pt]{amsart}

\usepackage[latin1]{inputenc}
\usepackage{latexsym,amsmath,amssymb,amscd,amsfonts,amsthm, mathrsfs}
\usepackage{array}
\usepackage{geometry}

\usepackage{color}

\usepackage[active]{srcltx}

\newtheorem{theorem}{Theorem}[section]
\newtheorem{lemma}[theorem]{Lemma}
\newtheorem{proposition}[theorem]{Proposition}
\newtheorem{corollary}[theorem]{Corollary}
\newtheorem{definition}[theorem]{Definition}
\newtheorem{example}[theorem]{\sc Example}
\newtheorem{remark}[theorem]{Remark}

\newcommand {\C}    {\mathbb{C}}

\newcommand {\R}    {\mathbb{R}}

\renewcommand{\epsilon}{\varepsilon}

\allowdisplaybreaks
\begin{document}
\date{\today}
\subjclass[2010]{Primary: 47B35; Secondary: 30H20, 81S10}
\keywords{Toeplitz operators, Fock space, semi-commutator, semi-classical limit, heat transform, oscillation}

\title[Toeplitz quantization on Fock space]{Toeplitz quantization on Fock space}

\author[W. Bauer]{W. Bauer}
\address{
Wolfram Bauer \endgraf 
Institut für Analysis \endgraf
Welfengarten 1, 30167 Hannover, Germany \endgraf}
\email{bauer@math.uni-hannover.de}

\author[L. A. Coburn]{L.A. Coburn}
\address{
Lewis A. Coburn  \endgraf
Department of Mathematics \endgraf
SUNY at Buffalo, New York 14260, USA \endgraf}
\email{lcoburn@buffalo.edu}

\author[R. Hagger]{R. Hagger}
\address{
Raffael Hagger  \endgraf
Institut für Analysis \endgraf
Welfengarten 1, 30167 Hannover, Germany \endgraf}
\email{raffael.hagger@math.uni-hannover.de}

\thanks{The first and third author acknowledge support through DFG (Deutsche Forschungsgemeinschaft), \\BA 3793/4-1.}

\begin{abstract}
For Toeplitz operators $T_f^{(t)}$ acting on the weighted Fock space $H_t^2$, we consider the semi-commutator $T_f^{(t)}T_g^{(t)} - T_{fg}^{(t)}$, where $t > 0$ is a certain weight parameter that may be interpreted as Planck's constant $\hbar$ in Rieffel's deformation quantization. In particular, we are interested in the semi-classical limit
\begin{equation}\tag{$*$}
\lim\limits_{t \to 0} \|T_f^{(t)}T_g^{(t)} - T_{fg}^{(t)}\|_t.
\end{equation}
It is well-known that $\|T_f^{(t)}T_g^{(t)} - T_{fg}^{(t)}\|_t$ tends to $0$ under certain smoothness assumptions imposed on $f$ and $g$. This result was recently extended to $f,g \in \textup{BUC}(\C^n)$ by Bauer and Coburn. We now further generalize $(*)$ to (not necessarily bounded) uniformly continuous functions and symbols in the algebra $\textup{VMO}\cap L^{\infty}$ of bounded functions  having vanishing mean oscillation on $\mathbb{C}^n$. Our approach is based on the algebraic identity $T_f^{(t)}T_g^{(t)} - T_{fg}^{(t)} = -(H_{\bar{f}}^{(t)})^*H_g^{(t)}$, where $H_g^{(t)}$ denotes the Hankel operator corresponding to the symbol $g$, and norm estimates in terms of the (weighted) heat transform. As a consequence, only $f$ (or likewise only $g$) has to be contained in one of the above classes for $(*)$ to vanish. For $g$ we only have to impose $\limsup_{t \to 0} \|H_g^{(t)}\|_t < \infty$, e.g.~$g \in L^{\infty}(\C^n)$. We prove that the set of all symbols $f\in L^{\infty}
 (\mathbb{C}^n)$ with the property that $\lim_{t \rightarrow 0}\|T^{(t)}_fT^{(t)}_g-T^{(t)}_{fg}\|_t= \lim_{t \rightarrow 0} \| T_g^{(t)}T_f^{(t)}-T_{gf}^{(t)}\|_t=0$ for all $g\in L^{\infty}(\mathbb{C}^n)$ coincides with $\textup{VMO}\cap L^{\infty}$. 
Additionally, we show that $\lim\limits_{t \to 0} \|T_f^{(t)}\|_t = \|f\|_{\infty}$ holds for all $f \in L^{\infty}(\C^n)$. Finally, we present new examples, including bounded smooth functions, where $(*)$ does not vanish. 
\end{abstract}

\maketitle
%%%%%%%%%%%%%%%%%%%%%%%%%%%%%%%%%%%%%%%%%%%%%%%%%%%%%%%%%%%%%%%%%%%%%%%%%
\section{Introduction}
\label{section_introduction}
\setcounter{equation}{0}
%%%%%%%%%%%%%%%%%%%%%%%%%%%%%%%%%%%%%%%%%%%%%%%%%%%%%%%%%%%%%%%%%%%%%%%%%%%

For a suitable family of functions $\mathcal{F}$ and Poisson bracket $\{\cdot,\cdot\}$, one can consider the deformation quantization (in the sense of Rieffel \cite{R,R2}) $f \mapsto T_f^{(t)}$, where $t \sim \hbar > 0$ is a weight parameter and $\{T_f^{(t)} : f \in \mathcal{F}\}$ consists of linear operators on an appropriate Hilbert space. Hereby essential are the following limit conditions:
\begin{align} \label{semiclassical_limit}
\lim\limits_{t \to 0} \| T^{(t)}_{f} \|_{t} &= \|f\|_{\infty}, \notag \\*
\lim\limits_{t \to 0} \|T_f^{(t)}T_g^{(t)} - T_{fg}^{(t)}\|_t &= 0,\\*
\lim\limits_{t \to 0} \|\frac{1}{it}[T_f^{(t)},T_g^{(t)}] - T_{\{f,g\}}^{(t)}\|_t &= 0.\notag
\end{align}
A typical approach to obtain such a quantization is to consider a family of weighted probability measures $\mu_t$ on some domain $\Omega \subseteq \C^n$ with corresponding Lebesgue space $L^2(\Omega,\mu_t)$ and then associate a Toeplitz operator $T_f^{(t)}$ to every $f \in \mathcal{F}$. This construction was considered for functions on several different domains $\Omega$ and results like \eqref{semiclassical_limit} have been obtained under certain smoothness assumptions on the functions $f$ and $g$ (see \cite{Be1,Be2,Be3,BMS,Bo,BLU,C,E,KL}).

In this paper we are interested in the case where $\Omega = \C^n$ and $\{\mu_t\}_{t > 0}$ is a family of Gaussian measures defined below. It has been shown in \cite{C} that \eqref{semiclassical_limit} holds in case $f$ and $g$ are the sum of trigonometric polynomials and $(2n+6)$-times continuously differentiable functions with compact support. In \cite{Bo} this result was extended to symbols $f$ and $g$ whose derivatives up to order $(4n+6)$ are continuous and bounded. It was shown recently in \cite{BC00} that the second equation of \eqref{semiclassical_limit} even holds for bounded uniformly continuous symbols (BUC).

The goal of the present paper is to push this result even further by considering (no longer bounded) uniformly continuous functions ($\textup{UC}(\mathbb{C}^n$)) as well the classical space \cite{S} of bounded functions of vanishing mean oscillation ($\textup{VMO}(\C^n) \cap L^{\infty}(\C^n)$). We note that the algebra $\textup{BUC}(\C^n)$ is contained in $\textup{VMO}(\C^n) \cap L^{\infty}(\C^n)$ but $\textup{VMO}(\C^n) \cap L^{\infty}(\C^n)$ is a well-studied algebra which also contains discontinuous functions.

The corresponding results for bounded symmetric domains have been obtained recently in \cite{BHV} and the methods of proof are rather similar. Consider the standard identity
\[T_f^{(t)}T_g^{(t)} - T_{fg}^{(t)} = -(H_{\bar{f}}^{(t)})^*H_g^{(t)},\]
where $H_{\cdot}^{(t)}$ denotes the corresponding Hankel operator. The idea now is to show that $\|H_{\bar{f}}^{(t)}\|_t$ tends to $0$ if $f$ is reasonably chosen. As a benefit of this approach, we only have to assume that $\|H_g^{(t)}\|_t$ is uniformly bounded near $0$ in order to obtain
\[\lim\limits_{t \to 0} \|T_f^{(t)}T_g^{(t)} - T_{fg}^{(t)}\|_t = 0.\]
In particular, $g$ can be chosen to be an arbitrary $L^{\infty}$-function. 
\vspace{1ex}\par 
Here is a short outline of our approach: In Section \ref{Section_Notation} we fix the notation. Let $f$ be a function of bounded mean oscillation.  We generalize a norm estimate on Hankel operators in 
\cite{B} to the family of Hankel operators $H_f^{(t)}$, $t>0$ acting on differently weighted Fock spaces. The main issue here is to choose the constant which appears in the norm estimate independently 
of the weight parameter $t$. Sections \ref{UC_symbols} and \ref{VMO_symbols} contain the proof of the semi-classial limit $(*)$ for uniformly continuous operator symbols and bounded symbols in 
$\textup{VMO}(\mathbb{C}^n)$, respectively. 
In Section \ref{Section_8} of the paper we consider the algebra $\textup{Op}(L)$ of all decomposable bounded operators $X=\oplus_{t>0}X^{(t)}$ acting on the direct integral $L:= \int_{t>0}^{\oplus}L_t^2$ of standard 
Gaussian weighted $L^2$-spaces. An ideal $\mathcal{I} \subset \textup{Op}(L)$ is defined by
\begin{equation*}
\mathcal{I}:= \Big{\{} X \in \textup{Op}(L) \: : \: \lim_{t \rightarrow 0} \| X^{(t)}\|_t =0 \Big{\}}. 
\end{equation*}
Theorem \ref{Main_theorem_Section_8} shows that the set $\mathcal{A}$ of all symbols $f \in L^{\infty}(\mathbb{C}^n)$ such that the semi-commutators 
$$T_fT_g-T_{fg}\hspace{3ex} \mbox{\it and}\hspace{3ex}  T_gT_f-T_{gf}, \hspace{3ex} T_g:= \oplus_{t>0} T_g^{(t)}$$
for all $g \in L^{\infty}(\mathbb{C}^n)$  belong to the ideal $\mathcal{I}$ is a closed and conjugate-closed subalgebra of $L^{\infty}(\mathbb{C}^n)$ and it precisely coincides with $\textup{VMO}(\mathbb{C}^n) \cap L^{\infty}(\mathbb{C}^n)$. 
\vspace{1mm}\par 
In Section \ref{new_Section_6} we show that the first equation of (\ref{semiclassical_limit}) holds for all $f \in L^{\infty}(\C^n)$. In Section \ref{examples} we provide some more examples of functions that 
do not satisfy the second equation of \eqref{semiclassical_limit} (a first example was already given in \cite{BC00}) as well as some further comments. 

%%%%%%%%%%%%%%%%%%%%%%%%%%%%%%%%%%%%%%%%%%%%%%%%%%%%%%%%%%%%%%%%%%%%%%%%%
\section{Notation and time dependent norm estimates}
\label{Section_Notation}
\setcounter{equation}{0}
%%%%%%%%%%%%%%%%%%%%%%%%%%%%%%%%%%%%%%%%%%%%%%%%%%%%%%%%%%%%%%%%%%%%%%%%%%%
Let $n\in \mathbb{N}$ and consider the Euclidean $n$-space $\mathbb{C}^n$. With $z,w \in \mathbb{C}^n$ we denote by $\langle z, w \rangle:= z_1\overline{w}_1 + \cdots + z_n \overline{w}_n$ and $|z|=\sqrt{\langle z,z\rangle}$ the Euclidean inner product and norm, respectively. For $t > 0$ we consider the following Gaussian probability measures
\[d\mu_t(z) = \frac{1}{(4\pi t)^n}e^{-\frac{|z|^2}{4t}} \, dv(z)\]
and the corresponding function space $L^2_t := L^2(\C^n,d\mu_t)$ on $\C^n$. The closed subspace of analytic functions in $L^2_t$ is denoted by $H^2_t$ and it 
forms a Hilbert space with reproducing and normalized reproducing kernel 
\begin{equation*}
K_t(z,w)= \exp \Big{\{} \frac{\langle z,w \rangle}{4t} \Big{\}} \hspace{3ex} \mbox{\it and} \hspace{3ex} k_w^t(z)= \frac{K_t(z,w)}{\| K_t( \cdot, w)\|}=  \exp \Big{\{} \frac{\langle z,w \rangle}{4t}- \frac{|w|^2}{8t} \Big{\}}. 
\end{equation*}
\par For a measurable function $f$ and an analytic function $g \in H^2_t$ with $fg \in L^2_t$, we write
\[T_f^{(t)}g := P^{(t)}(fg) \quad \mbox{\it and} \quad H_f^{(t)}g := (I - P^{(t)})(fg)\]
for the corresponding Toeplitz and Hankel operators, respectively, where $P^{(t)}$ denotes the orthogonal projection from $L^2_t$ onto $H^2_t$. Note that $P^{(t)}$ can be expressed as the integral operator: 
\begin{equation*}
\big{[}P^{(t)}h\big{]}(z)=\int_{\mathbb{C}^n}h(u) K_t(z,u) d\mu_t(u), \hspace{4ex} h \in L_t^2. 
\end{equation*}
On $L_t^2$ we have the usual inner product 
\begin{equation*}
\langle f,g \rangle_t= \int_{\mathbb{C}^n} f \overline{g} d\mu_t \hspace{3ex} \mbox{\it with} \hspace{3ex} \| f\|_t^2= \langle f,f \rangle_t. 
\end{equation*}
For any linear operator $X$ with domain and range in $L_t^2$ we have the usual operator norm $\|X\|_t= \sup \{ \|Xf\|_t/ \|f\|_t \: : \: f \ne 0 \}$. For bounded (or essentially bounded) functions $f$ we write $\| f\|_{\infty}$ for the 
supremum (or essential supremum) of $|f|$. 
\vspace{1ex}\par 
In some of our results we allow operator symbols $f$ in the space $\textup{UC}(\mathbb{C}^n)$ which contains unbounded functions. Then the Toeplitz operator $T_f^{(t)}$  may be unbounded as well. 
Hence we have to specify the domain and carefully define operator products. This issue is discussed in \cite{B0} where a function space 
$\textup{Sym}(\mathbb{C}^n)$ and an increasing scale $(\mathcal{H}_t^n)_{n\in \mathbb{N}}$ of Hilbert spaces in $L_t^2$ are defined. It is shown that $P^{(t)}$ and the multiplication $M_f$ for all 
$f \in  \textup{Sym}(\mathbb{C}^n)$ are operators acting on this scale. In particular,  finite products of these operators are well-defined. One easily checks that $\textup{Sym}(\mathbb{C}^n)$ contains the space 
$\textup{UC}(\mathbb{C}^n)$ and therefore we can form finite products of Toeplitz operators with uniformly continuous symbols. 
\vspace{1mm}\par 
We will also need the heat transform
\begin{equation}\label{heat_transform}
\tilde{f}^{(t)}(w) := \frac{1}{(4\pi t)^n} \int_{\C^n} f(w-z)e^{-\frac{|z|^2}{4t}} \, dv(z) = \int_{\C^n} f(w-z) \, d\mu_t(z).
\end{equation}
\begin{remark} \label{heat_remark}
{\rm 
Note that $\tilde{f}^{(t)}(w)= \langle T_f^{(t)} k_w^t, k_w^t \rangle_t$ for $w \in \mathbb{C}^n$ so that $\|\tilde{f}^{(t)}\|_{\infty} \leq \| T_f^{(t)} \|_t$ by the Cauchy-Schwarz inequality. }
\end{remark}
The {\it mean oscillation} of a function $f$ is given by
\begin{equation} \label{Property_and_defn_of_MO}
\textup{MO}^t(f)(w) = (|f|^2)^{\sim(t)}(w) - |\tilde{f}^{(t)}(w)|^2 = \int_{\C^n} |f(w-z) - \tilde{f}^{(t)}(w)|^2 \, d\mu_t(z).
\end{equation}
We will say that a function $f$ has {\it bounded mean oscillation} if the semi-norm
\[\|f\|_{\textup{BMO}_{\ast}^t} := \sup\limits_{w \in \C^n} \sqrt{\textup{MO}^t(f)(w)}\]
is finite. The (linear) space of all functions having finite $\textup{BMO}_{\ast}^t$ semi-norm is denoted by 
\begin{equation}\label{defn_BMO_ast}
\textup{BMO}_{\ast}(\mathbb{C}^n):= \Big{\{} f \in L^1_{\textup{loc}}(\mathbb{C}^n) \: : \: \| f \| _{\textup{BMO}_{\ast}^t} < \infty \Big{\}}. 
\end{equation}
\par 
Recall that the right hand side (as a vector space) does not depend on the parameter $t$ (see \cite{BCI}) and therefore we do not indicate $t$ in the notation. 
A different, more standard version of $\textup{BMO}_{\ast}(\C^n)$ is considered in Section \ref{VMO_symbols}. As is well-known $\textup{BMO}_{\ast}(\C^n)$ contains unbounded functions. Hankel operators with symbols in 
$\textup{BMO}_{\ast}(\C^n)$ are considered in \cite{B} and it is shown that they are bounded on a dense domain. We mention that our  notations of the spaces $\textup{BMO}_{\ast}(\C^n)$, 
$\textup{BMO}(\mathbb{C}^n)$ and $\textup{VMO}(\mathbb{C}^n)$ in Section \ref{VMO_symbols} are different from the ones in \cite{B}. 
%%%%%%%%%%%%%%%%%%%%%%%%%%%%%%%%%%%%%%%%%%%%%%%%%%%%%%%%%%%%%%%%%%%%%%%%%
%\section{Time dependent norm estimates}
%\label{UNE_for_TO_and Hankel}
%\setcounter{equation}{0}
%%%%%%%%%%%%%%%%%%%%%%%%%%%%%%%%%%%%%%%%%%%%%%%%%%%%%%%%%%%%%%%%%%%%%%%%%%%
\vspace{1ex}\par
In the remaining part of this section we generalize a norm estimate for Hankel operators in \cite{B} to operators acting on the above family of Hilbert spaces. 
For each $t>0$ consider the operator 
\begin{equation}\label{operator_U_t}
U_t: L^2_t \rightarrow L^2_{\frac{1}{4}}: f \mapsto \big{[}U_tf\big{]}(z):= f(z2\sqrt{t}). 
\end{equation}
A simple calculation shows that $U_t$ is an isometry onto $L^2_{\frac{1}{4}}$ with inverse $U_t^{-1}= U_{\frac{1}{16 t}}$.  Moreover, $U_t$  restricts to an isometry from $H^2_t$ onto $H^2_{\frac{1}{4}}$ and for all 
$f \in L^{\infty}(\mathbb{C}^n)$ one has: 
$$U_t M_f=M_{f(\cdot 2 \sqrt{t})} U_t \hspace{6ex} \mbox{\it and } \hspace{6ex} U_tP^{(t)}=P^{(\frac{1}{4})} U_t.$$
\par 
Clearly, the first equality remains true for an unbounded operator symbol $f$ when we restrict the product of operators to the maximal domain of $M_f$. Combining these two relations gives 
\begin{align}
U_tT_f^{(t)} U_t^{-1}
&= T^{(\frac{1}{4})}_{f(\cdot 2 \sqrt{t})}, \label{Transformation_Toeplitz_operator}\\
U_t H_f^{(t)}U_t^{-1} 
&= H^{(\frac{1}{4})}_{f(\cdot 2\sqrt{t})}. \label{transformation_Hankel_Toeplitz_operator}
\end{align}
The following result is Theorem 4.2 in \cite{B}: 
\begin{theorem}\label{theorem_boundedness_Hankel_operator_t_=_1_4}(\cite{B}) 
Let $f \in \textup{BMO}_{\ast}(\mathbb{C}^n)$. Then the Hankel operator $H_f^{(\frac{1}{4})}$ is bounded and there is a constant $C>0$ independent of $f$ such that 
\begin{equation}\label{Norm_estimate_BMO_1_4}
\|H_f^{(\frac{1}{4})}\|_{\frac{1}{4}} \leq C \| f\|_{\textup{BMO}_{\ast}^{\frac{1}{4}}}. 
\end{equation}
\end{theorem}
We can use the family of isometries $(U_t)_{t >0}$ between the differently weighted $L^2$-spaces to generalize the inequality (\ref{Norm_estimate_BMO_1_4}) with a $t$-independent constant $C$. 
Note that for $z \in \mathbb{C}^n$: 
\begin{equation*}
\Big{\{} f(\cdot 2 \sqrt{t}) \widetilde{\Big{\}}}^{(\frac{1}{4})}(z)=% \frac{1}{\pi^n} \int_{\mathbb{C}^n} f(u 2\sqrt{t}) e^{-|z-u|^2} dv(u)=
 \tilde{f}^{(t)}\big{(}z 2 \sqrt{t}\big{)} 
\end{equation*}
and therefore we obtain the identity: 
\begin{equation} \label{identity_BMO_new}
\|f(\cdot 2\sqrt{t})\|_{\textup{BMO}_{\ast}^{\frac{1}{4}}} = \|f\|_{\textup{BMO}_{\ast}^t}.
\end{equation}
\begin{corollary}\label{main_theorem_section_2} 
Let $f \in \textup{BMO}_*(\mathbb{C}^n)$. Then there is a constant $C>0$ independent of $t>0$ and $f$ such that 
\begin{equation}
\label{uniform_estimate_Hankel_operators}
\|H_f^{(t)}\|_t \leq C \|f\|_{\textup{BMO}_{\ast}^t}. 
\end{equation}
\end{corollary}
\begin{proof}
We combine Theorem \ref{theorem_boundedness_Hankel_operator_t_=_1_4} with (\ref{transformation_Hankel_Toeplitz_operator}) and (\ref{identity_BMO_new}): 
\begin{align*}
\|H_f^{(t)} \|_t 
&= \| U_t H_f^{(t)} U_t^{-1}\|_{\frac{1}{4}}
= \| H_{f(\cdot 2 \sqrt{t})}^{(\frac{1}{4})}\|_{\frac{1}{4}}\leq C \| f(\cdot 2 \sqrt{t})\|_{\textup{BMO}_{\ast}^{\frac{1}{4}}}=C\|f\|_{\textup{BMO}_{\ast}^t}, 
\end{align*}
where $C>0$ is the constant in Theorem \ref{theorem_boundedness_Hankel_operator_t_=_1_4} which is independent of $t$ and $f$. 
\end{proof}
%%%%%%%%%%%%%%%%%%%%%%%%%%%%%%%%%%%%%%%%%%%%%%%%%%%%%%%%%%%%%%%%%%%%%%%%%
\section{Operators with uniformly continuous symbols}
\label{UC_symbols}
\setcounter{equation}{0}
%%%%%%%%%%%%%%%%%%%%%%%%%%%%%%%%%%%%%%%%%%%%%%%%%%%%%%%%%%%%%%%%%%%%%%%%%%%
In the following we denote by $\textup{UC}(\mathbb{C}^n)$ the space of all uniformly continuous complex valued functions on $\mathbb{C}^n$. Moreover, we write 
$$\textup{BUC}(\mathbb{C}^n)= \textup{UC}(\mathbb{C}^n) \cap L^{\infty}(\mathbb{C}^n)$$
for the $C^*$ algebra of all bounded uniformly continuous functions. The next two propositions recall some results from \cite{BC1}:

\begin{proposition} \label{Lipschitz_continuity_heat_transform}
Let $f \in \textup{UC}(\C^n)$. Then $f$ is in $\textup{BMO}_{\ast}(\C^{n})$, $\tilde{f}^{(t)}$ is Lipschitz continuous and $f - \tilde{f}^{(t)}$ is bounded. Moreover, the Lipschitz constant of $\tilde{f}^{(t)}$ is 
bounded by $\frac{1}{\sqrt{t}}\|f\|_{\textup{BMO}_{\ast}^t}$.
\end{proposition}

\begin{proof}
See \cite[Lemma 2.1]{BC1} and \cite[Proposition 3.1]{BC1} for the first two statements. In order to estimate the Lipschitz constant of $\tilde{f}^{(t)}$ recall that in the proof of 
\cite[Proposition 3.1]{BC1} (see also \cite[Corollary 2.7]{B}) it was shown that for all $t>0$: 
\begin{equation}\label{Lipschitz_estimat_f_tilde_t}
\big{|} \tilde{f}^{(t)}(z)-\tilde{f}^{(t)}(w)\big{|} \leq t^{-\frac{1}{2}}\| f(\cdot 2\sqrt{t})\|_{\textup{BMO}_{\ast}^{\frac{1}{4}}} |z-w|=t^{-\frac{1}{2}} \|f\|_{\textup{BMO}_{\ast}^t} |z-w|. 
\end{equation}
\par 
In the last equality we have use the identity (\ref{identity_BMO_new}) again.  
\end{proof}

\begin{proposition} \label{Lemma_convergence_heat_transform_UC_functions}
Let $f \in \textup{UC}(\C^n)$. Then $\tilde{f}^{(t)} \to f$ uniformly on $\C^n$. 
\end{proposition}

\begin{proof}
See \cite[Proposition 3.2]{BC1}. 
\end{proof}

\begin{proposition} \label{Lemma_limit_BMO_lambda_UC}
Let $f \in \textup{UC}(\C^n)$. Then $\|f\|_{\textup{BMO}_{\ast}^t} \to 0$ as $t \to 0$. 
\end{proposition}

\begin{proof}
Let $\frac{1}{4} > \epsilon > 0$ be fixed and choose $\delta >0$ such that $|f(z)-f(w)| < \epsilon$ for all $z,w \in \C^n$ with $|z-w| < \delta$. We divide the domain of integration into two parts: 
\begin{equation} \label{decomposition_MO_lambda}
\textup{MO}^t(f)(w)
= \Big{\{}\int_{|z| < \delta} + \int_{|z| \geq \delta} \Big{\}}|f(w-z) - \tilde{f}^{(t)}(w)|^2 \, d\mu_t(z).
\end{equation}
By Proposition \ref{Lemma_convergence_heat_transform_UC_functions}, we have $\|f - \tilde{f}^{(t)}\|_{\infty} \leq \epsilon$ for sufficiently small $t$. Thus the integrand can be estimated as follows:
\begin{align*}
|f(w-z) - \tilde{f}^{(t)}(w)|^2 
&\leq \left(|f(w-z) - f(w)| + |f(w) - \tilde{f}^{(t)}(w)|\right)^2 \\
&\leq \left(|f(w-z) - f(w)| + \epsilon\right)^2.
\end{align*}

In the case $|z| < \delta$ we get $|f(w-z) - f(w)| < \epsilon$ and thus
\[\int_{|z| < \delta} |f(w-z) - \tilde{f}^{(t)}(w)|^2 \, d\mu_t(z) \leq 4\epsilon^2 < \epsilon\]
for sufficiently small $t$.

In the case $|z| \geq \delta$ we observe that $f$ is the sum of a bounded and a Lipschitz continuous function (cf.~Proposition \ref{Lipschitz_continuity_heat_transform}) and hence 
$|f(w-z) - f(w)| \leq C(1+|z|)$ for a suitable constant $C$. Now,
\begin{align*}
\int_{|z| \geq \delta} |f(w-z) - \tilde{f}^{(t)}(w)|^2 \, d\mu_t(z) &\leq \frac{1}{(4\pi t)^n} \int_{|z| \geq \delta} (C(1+|z|) + \epsilon)^2 e^{-\frac{|z|^2}{4t}} \, dv(z)\\
&\leq \frac{1}{(4\pi t)^n}e^{-\frac{\delta^2}{8t}} \int_{|z| \geq \delta} (C(1+|z|) + \epsilon)^2 e^{-|z|^2} \, dv(z)\\
\end{align*}
for $t \leq \frac{1}{8}$. As the integral on the right-hand side is bounded, the whole term converges to $0$ as $t \to 0$. We conclude $\textup{MO}^t(f)(w) \to 0$ uniformly in $w$. Thus the assertion follows.
\end{proof}
 
\begin{theorem} \label{Theorem_behaviour_semi_commutator_uniformly_bounded_function}
Let $f \in \textup{UC}(\C^n)$. Then $\lim\limits_{t \to 0} \|H_f^{(t)}\|_t = 0$. In particular,  
\begin{equation}\label{limit_semi_commutator}
\lim\limits_{t \to 0} \big{\|}T_f^{(t)}T_g^{(t)} - T_{fg}^{(t)}\big{\|}_t = 0 
\end{equation}
for all $g \in L^{\infty}(\C^n)$ or all $g \in \textup{UC}(\C^n)$. 
\end{theorem}

\begin{proof}
Because of $T_f^{(t)}T_g^{(t) }- T_{fg}^{(t)} = -(H_{\bar{f}}^{(t)})^*H_g^{(t)}$, it is sufficient to show 
\begin{equation}\label{limit_Hankel_operator_UC}
\lim\limits_{t \to 0} \| H_f^{(t)}\|_t = 0.
\end{equation}
This follows from Corollary \ref{main_theorem_section_2} in combination with Proposition \ref{Lemma_limit_BMO_lambda_UC}. 
\end{proof}

%\begin{remark} \label{proof_remark}
%Instead of Corollary \ref{main_theorem_section_2}, one may also use Proposition \ref{estimate_norm_Hankel_operator} and Proposition \ref{Lipschitz_continuity_heat_transform} in addition to Proposition %\ref{Lemma_convergence_heat_transform_UC_functions} and 
%Proposition \ref{Lemma_limit_BMO_lambda_UC} to arrive at the conclusion:
%\[\|H_f^{(t)}\|_t \leq \|H_{f-\tilde{f}^{(t)}}^{(t)}\|_t + \|H_{\tilde{f}^{(t)}}^{(t)}\|_t \leq \|f-\tilde{f}^{(t)}\|_{\infty} + C\|f\|_{\textup{BMO}_{\ast}^t} \to 0.\]
%\end{remark}

Another interesting question is: For which symbols $f$ and $g$ are the semi-commutators $T_f^{(t)}T_g^{(t)} - T_{fg}^{(t)}$ in \eqref{limit_semi_commutator} compact? 
In other words, what is the maximal algebra $\mathcal{T}$ generated by Toeplitz operators so that the Calkin algebra $\mathcal{T} / \mathcal{K}(H_t^2)$ is commutative? 
For $\textup{BUC}$-symbols we will answer these questions in Proposition \ref{compact_semi_commutators} below. Therefore we need the following notions:

For a bounded and continuous function $f: \C^n \to \C$ (in short: $f \in \textup{BC}(\C^n)$) define
\[\textup{Osc}_z(f) := \sup\{|f(z)-f(w)| : |z-w| < 1\}.\]
The space of functions having vanishing oscillation is 
$$\textup{VO}(\C^n) := \big{\{}f \in \textup{BC}(\C^n) : \textup{Osc}_z(f) \to 0 \text{ as } |z| \to \infty\big{\}}.$$ 
Note that in fact $\textup{VO}(\C^n) \subset \textup{BUC}(\C^n)$. Indeed, for $f \in \textup{VO}(\C^n)$ and $\epsilon > 0$ we may choose a compact subset $K \subset \C^n$ such that $\textup{Osc}_z(f) < \epsilon$ for all $z \in \C^n \setminus K$. As $f$ is uniformly continuous on compact sets, there is a $\delta_K$ so that $|f(z)-f(w)| < \epsilon$ for all $z \in K,w \in \C^n$ with $|z-w| < \delta_K$. Thus $|f(z)-f(w)| < \epsilon$ for all $z,w \in \C^n$ with $|z-w| < \min\{\delta_K,1\}$.

Here is one possible way of constructing functions in $\textup{VO}(\C^n)$:

\begin{example}
Choose $f \in \textup{BUC}(\R)$ and compose it with a function $g : \C^n \to \R$ that is continuous on $\C^n$, differentiable in a neighborhood of $\infty$ and its derivative $g'$ tends to $0$ as $|z| \to \infty$. To show that $f \circ g$ is in $\textup{VO}(\C^n)$, choose $\epsilon > 0$ and select $\delta > 0$ such that $|f(x)-f(y)| < \epsilon$ for $x,y \in \R$ with $|x-y| < \delta$. By the mean value theorem we may choose $R > 0$ sufficiently large such that $|g(z) - g(w)| < \delta$ for all $z,w \in \C^n$ with $|z| > R$ and $|z-w| < 1$. This implies $|(f \circ g)(z) - (f \circ g)(w)| < \epsilon$ and hence $\textup{Osc}_z(f \circ g) \to 0$ as $|z| \to \infty$. We conclude $f \circ g \in \textup{VO}(\C^n)$.

For a concrete example consider $z \mapsto \exp(i\sqrt{|z|})$.
\end{example}

With these preparations we can now prove the following proposition.

\begin{proposition} \label{compact_semi_commutators}
Let $f \in \textup{BUC}(\C^n)$ and $t > 0$. Then the following are equivalent:
\begin{itemize}
	\item[(i)] $f \in \textup{VO}(\C^n)$,
	\item[(ii)] $H_f^{(t)}$ and $H_{\bar{f}}^{(t)}$ are compact,
	\item[(iii)] $T_f^{(t)}T_g^{(t)} - T_{fg}^{(t)}$ and $T_g^{(t)}T_f^{(t)} - T_{gf}^{(t)}$ are compact for all $g \in L^{\infty}(\C^n)$.
\end{itemize}
\end{proposition}

\begin{proof}
Let $f \in \textup{VO}(\C^n)$. Then $H_f^{(t)}$ and $H_{\bar{f}}^{(t)}$ are compact by a straight forward extension of \cite[Theorem 5.3]{B} from the case $t = \frac{1}{4}$ to general $t > 0$.

Now assume that $H_f^{(t)}$ and $H_{\bar{f}}^{(t)}$ are both compact. Then $\tilde{f}^{(t)} \in \textup{VO}(\C^n)$ by \cite[Theorem 5.3]{B} again and $T_{f-\tilde{f}^{(t)}}^{(t)}$ is compact by \cite[Theorem 3.1]{B}. Moreover, $f-\tilde{f}^{(t)} \in \textup{BUC}(\C^n)$ and thus $f-\tilde{f}^{(t)} \in C_0(\C^n) \subset \textup{VO}(\C^n)$ by \cite[Theorem 2.3]{BC0}. As the sum of two functions in $\textup{VO}(\C^n)$ is obviously again in $\textup{VO}(\C^n)$, we conclude that $f$ is in $\textup{VO}(\C^n)$.
\vspace{1ex}\par 
That (ii) and (iii) are equivalent is standard.
\end{proof}
%%%%%%%%%%%%%%%%%%%%%%%%%%%%%%%%%%%%%%%%%%%%%%%%%%%%%%%%%%%%%%%%%%%%%%%%%
\section{Symbols in VMO}
\label{VMO_symbols}
\setcounter{equation}{0}
%%%%%%%%%%%%%%%%%%%%%%%%%%%%%%%%%%%%%%%%%%%%%%%%%%%%%%%%%%%%%%%%%%%%%%%%%%%%%%%%%%%%
In the present section we consider the quantization problem for operators with symbols of vanishing mean oscillation. We start by recalling some notation. Consider locally 
integrable functions $f: \mathbb{R}^{n} \rightarrow \mathbb{C}$ with average value
\begin{equation*}
f_{E} = {1\over |E|} \int_{E} f
\end{equation*}
on a bounded measurable subset $E\subset \mathbb{R}^n$ with finite measure $|E|$.  We consider the variance of $f$ on $E$
\begin{equation*}
\text{Var}_{E} (f) = {1\over |E|} \int_{E} |f - f_{E}|^{2}
\end{equation*}
as well as the corresponding quantity
\begin{equation*}
\textup{Osc}_{E} (f) = {1\over |E|} \int_{E} |f - f_{E} |.
\end{equation*}

\begin{definition} (see \cite{S,ST}) We say $f$ is in $\textup{BMO}(\mathbb{R}^{n})$ if the set  
$$\big{\{}\textup{Osc}_{E}(f): E \text{ each n-cube in } \mathbb{R}^{n}\big{\}}$$ 
is bounded. We say $f$ is in $\textup{VMO}(\mathbb{R}^{n})$ if $f$ is in $\textup{BMO}(\mathbb{R}^{n})$ and, for all $n$-cubes $E$,
\begin{equation*}
\lim_{a \rightarrow 0}  \sup\big{\{} \textup{Osc}_{E}(f): |E|\leq a\big{\}} = 0.
\end{equation*}
If we replace $\textup{Osc}_{E}(f)$ by $\textup{Var}_{E}(f)$, we get new sets $\textup{BMO}_{2}(\mathbb{R}^n)$ and $\textup{VMO}_{2}(\mathbb{R}^n)$.  
\end{definition}

\begin{remark}
$\textup{BMO}(\R^{2n})$ and $\textup{BMO}_{\ast}(\C^n)$ defined earlier are quite different.  In particular, $\textup{UC}(\R^{2n}) \cong \textup{UC}(\C^n)$ is not contained in $\textup{BMO}(\R^{2n})$.
\end{remark}

The Cauchy-Schwarz inequality shows that 
\begin{equation*}
\textup{BMO}_{2}(\mathbb{R}^n) \subset \textup{BMO}(\mathbb{R}^n), \hskip.2in \textup{VMO}_{2}(\mathbb{R}^n)\subset \textup{VMO}(\mathbb{R}^n).
\end{equation*}
\noindent
Direct calculation shows (eg. \cite[p. 313]{BBCZ})

\begin{lemma}
\label{lemma_representation_ Var}
We have, for arbitrary bounded measurable subsets $E$,
\begin{equation*}
\textup{Var}_{E}(f) = {1 \over{2|E|^{2}}} \int_{E} \int_{E}|f(z) - f(w)|^{2} dv(z) dv(w)
\end{equation*}
so that, for $F \subset E$,
\begin{equation*}
\textup{Var}_{E}(f) \geq {|F|^{2} \over |E|^{2}} \textup{Var}_{F}(f).
\end{equation*}
\end{lemma}

\begin{remark}\label{remark_cubes_ball}
{\rm Because of Lemma \ref{lemma_representation_ Var}, $\textup{BMO}_{2}(\mathbb{R}^n)$ and $\textup{VMO}_{2}(\mathbb{R}^n)$ have the additional property 
that $n$-cubes can be replaced by $n$-balls in their definitions. We need only consider the inscribed and circumscribed balls for a given cube. }
\end{remark}

\vspace{1ex} \par 
In the analysis that follows, we identify $\mathbb{C}^{n}$ with $\mathbb{R}^{2n}$ in the standard way and we consider the set $\text {VMO}(\mathbb{C}^{n}) \cap L^{\infty}(\mathbb{C}^{n})$, 
the essentially bounded functions in $\text{VMO}(\mathbb{C}^{n})$. It is easy to see that $\textup{BUC}(\mathbb{C}^n) \subset \textup{VMO}_2(\mathbb{C}^n)$. 
However, there are discontinuous functions in $\textup{VMO}(\mathbb{C}^n) \cap L^{\infty}(\mathbb{C}^{n})$, \cite[p.290]{ST}. We have: 
 %An easy estimate using $|f - f_{E}| \leq 2 \|f\|_{\infty}$ now shows that
%\begin{equation*}
%\textup{VMO}(\mathbb{C}^n) \cap L^{\infty} = \textup{VMO}_{2}(\mathbb{C}^n) \cap L^{\infty} .
%\end{equation*}
%It is easy to see that $\textup{UC}(\mathbb{C}^{n}) \subset \textup{VMO}_{2}(\mathbb{C}^n)$ and $\textup{VMO}(\mathbb{C}^n) \cap L^{\infty}$  contains discontinuous functions 
%(cf. $f(x)= \log \circ \log (1+|x|^{-1})$). As an application of Lemma \ref{lemma_representation_ Var}, we have
\begin{theorem}\label{theorem_VMO_algebra}
$\textup{VMO}(\mathbb{C}^n) \cap L^{\infty}(\mathbb{C}^{n})$ is a sup norm-closed, conjugate closed subalgebra of $L^{\infty}(\mathbb{C}^{n})$ with 
\begin{equation}\label{Equality_of_VMO_spaces_intersection}
\textup{VMO}(\mathbb{C}^n) \cap L^{\infty}(\mathbb{C}^{n}) = \textup{VMO}_2(\mathbb{C}^n) \cap L^{\infty}(\mathbb{C}^{n}). 
\end{equation}
\end{theorem}
\begin{proof} A easy estimate using $|f-f_E| \leq 2 \| f\|_{\infty}$ implies (\ref{Equality_of_VMO_spaces_intersection}). 
Integrating the inequality
\begin{equation*}
|f(z) g(z) - f(w) g(w)|^{2}\leq 2 \|f\|_{\infty}^{2} |g(z) -g(w)|^{2} + 2\|g\|_{\infty}^{2} |f(z) - f(w)|^{2}
\end{equation*}
over $E \times E$ shows that for $f,g$ in $\textup{VMO}_{2}(\mathbb{C}^n) \cap L^{\infty}(\mathbb{C}^{n})$, we also have $fg$ in $\textup{VMO}_{2}(\mathbb{C}^n) \cap L^{\infty}(\mathbb{C}^{n})$. 
The remainder of the proof is standard.
\end{proof}
\begin{remark}
{\rm 
Theorem \ref{theorem_VMO_algebra} gives, for  $\mathbb{C}^{n}$ (or $\mathbb{R}^{n}$), a classical result of Sarason \cite{S} for the circle.}
\end{remark}
Now we consider the quantization problem for Toeplitz operators with bounded symbols in $\textup{VMO}(\mathbb{C}^n)$. In the definitions above and for convenience we pass from $n$-cubes 
$E \subset \mathbb{C}^n$ to Euclidean $n$-balls 
\begin{equation*}
B(x, \rho):= \{ z \in \mathbb{C}^n \: : \: |z-x| < \rho\}
\end{equation*}
centered at $x \in \mathbb{C}^n$ and with radius $\rho>0$, cf. Remark \ref{remark_cubes_ball}. 
%%%%%%%%%%%%%%%%%%%%%%%%%%%%%%%%%%%%%%%%%%%%%%%%%%%%%%%%%%%%%%%%%%%%%%%%%%%%%%
%With $x \in \mathbb{C}^n$  and $\rho >0$ we consider the Euclidean balls $B(x, \rho):= \{ z \in \mathbb{C}^n \: : \: |z-x| < \rho\}$ with volume 
%\begin{equation*}
%|B(x, \rho)| = \int_{B(x,\rho)} 1 dv(z). 
%\end{equation*}
%Let $f$ be a locally integrable function on $\mathbb{C}^n$ (we write $f \in L^1_{\textup{loc}}(\mathbb{C}^n)$. Then we define the averaging function 
%\begin{equation*}
%\hat{f}(x, \rho):= \frac{1}{|B(x,\rho)|} \int_{B(x, \rho)} f(z) dv(z).
%\end{equation*}
%Let $q \geq 1$ and $f \in L^q(\mathbb{C}^n)$. Then we put 
%\begin{equation*}
%A_q(f, \rho,x) = \frac{1}{|B(x,\rho)|} \int_{B(x,\rho)} |f(z)-\hat{f}(x,\rho)|^q dv(z). 
%\end{equation*}
%The space of bounded functions that have vanishing oscillation in $\mathbb{C}^n$ is defined by
%\begin{equation*}
%\textup{VMO}(\mathbb{C}^n):= \Big{\{} f \in L^{\infty}(\mathbb{C}^n) \: : \: \lim_{\rho \rightarrow 0} A_2(f,\rho,x)=0 \; \mbox{\it uniformly for $x \in \mathbb{C}^n$} \Big{\}}. 
%\end{equation*}
\begin{lemma}\label{estimate_hilf_mean_oscillation}
Let $f\in L^{\infty}(\mathbb{C}^n)$ and $x \in \mathbb{C}^n$. Then, for all $\rho >0$ and all $t>0$ we have 
\begin{equation*}
\textup{MO}^t(f)(x) \leq \int _{\mathbb{C}^n} \big{|}f(x-z)-f_{B(x,\rho)}\big{|}^2 d\mu_t(z).
\end{equation*}
\end{lemma}
\begin{proof}
Consider the function $L:\mathbb{C} \rightarrow [0, \infty)$ defined by 
\begin{equation*}
L(c)= \int_{\mathbb{C}^n} \big{|}f(x-z) -c\big{|}^2 d\mu_t(z). 
\end{equation*}
Standard arguments show that $L$ attains a global minimum at $c= \tilde{f}^{(t)}(x)$. Hence the lemma follows from  $\textup{MO}^t(f)(x)= L \circ \tilde{f}^{(t)}(x)$. 
\end{proof}
\begin{theorem}
\label{Theorem_VMO_vanishing_BMO}
Let $f \in \textup{VMO}(\mathbb{C}^n)\cap L^{\infty}(\mathbb{C}^{n})$, then $\lim_{t \rightarrow 0} \| f\|_{\textup{BMO}_{\ast}^t}=0$.
\end{theorem}
\begin{proof}
Let $t>0$ and fix a parameter $\alpha >0$ which we will specify later on. According to Lemma \ref{estimate_hilf_mean_oscillation} we can estimate the mean oscillation of $f$ as follows: 
\begin{equation*}
\textup{MO}^t(f)(x)
\leq \Big{\{} \int_{ |z| \leq \alpha \sqrt{t}} + \int_{|z| > \alpha \sqrt{t}} \Big{\}} \big{|}f(x-z)-f_{B(x, \alpha \sqrt{t})} \big{|}^2 d\mu_t(z). 
\end{equation*}
We denote the first and the second integral by $I_{1,t,\alpha}(x)$ and $I_{2,t,\alpha}(x)$, respectively, and estimate them separately. By the Cauchy-Schwarz inequality we have
\begin{multline*}
I_{1,t,\alpha}(x)= \frac{1}{(4\pi t)^n}\int_{|x-z| \leq \alpha \sqrt{t}} \big{|} f(z)- f_{B(x, \alpha \sqrt{t})} \big{|}^2 e^{-\frac{|x-z|^2}{4t}} dv(z)\\
% \exp \Big{\{} \frac{1}{2t}\:  \textup{Re}(\langle z,x\rangle) - \frac{|x|^2}{4t}\Big{\}}  d\mu_t(z)\\
 \leq \frac{1}{(4 \pi t)^n} \left\{ \int_{B(x, \alpha \sqrt{t})} e^{- \frac{|x-z|^2}{2t}} dv(z) \right\}^{\frac{1}{2}} \left\{ \int_{B(x, \alpha \sqrt{t})} |f(z)-f_{B(x, \alpha \sqrt{t})}|^4 dv(z) \right\}^{\frac{1}{2}}. 
\end{multline*}
The first integral takes the form 
\begin{equation*}
 \int_{B(x, \alpha \sqrt{t})} e^{- \frac{|x-z|^2}{2t}} dv(z)= (2t)^n C_{\alpha} \hspace{3ex} \mbox{\it where } \hspace{3ex} C_{\alpha}= \int_{B(0, \alpha)} e^{-|z|^2} dv(z). 
\end{equation*}
Therefore 
\begin{equation*}
I_{1,t, \alpha}(x) \leq 2\|f\|_{\infty}  \frac{\sqrt{C_{\alpha}}}{(8\pi^2t)^{\frac{n}{2}}} \sqrt{|B(x, \alpha \sqrt{t})|}  \cdot \sqrt{\textup{Var}_{B(x, \alpha\sqrt{t})}(f)}. 
\end{equation*}
Since $t^{-\frac{n}{2}} \sqrt{|B(x, \alpha \sqrt{t})|} = O(1)$ as $t \rightarrow 0$ and $\lim_{t \rightarrow 0}\sup \{\textup{Var}_{B(x, \alpha\sqrt{t})}(f)\: : \: x \in \mathbb{C}^n\}=0$ it follows that 
\begin{equation}\label{limit_first_integral}
\lim_{t \rightarrow 0} I_{1,t, \alpha}(x) =0 
\end{equation}
uniformly for $x \in \mathbb{C}^n$. Next we estimate the second integral $I_{2,t,\alpha}(x)$. We obtain
\begin{equation}\label{limit_second_integral}
I_{2,t,\alpha}(x)\leq 4 \|f\|_{\infty}^2 \int_{|z|> \alpha \sqrt{t}} d\mu_t(z)= \frac{4 \|f\|^2_{\infty}}{(4\pi)^n}\int_{|z|> \alpha} e^{-\frac{|z|^2}{4}} dv(z). 
\end{equation}
\par 
Let $\varepsilon >0$ and with (\ref{limit_second_integral}) choose $\alpha >0$ sufficiently large such that $0 \leq I_{2,t,\alpha}(x)<\varepsilon$. With this fixed $\alpha >0$ and 
(\ref{limit_first_integral}) we can choose $t>0$ sufficiently small such that $I_{1,t,\alpha}(x)< \varepsilon$ for all $x \in \mathbb{C}^n$. Then 
\begin{equation*}
\|\textup{MO}^t(f)\|_{\infty} <2 \varepsilon
\end{equation*}
and the assertion follows. 
\end{proof}
\begin{theorem}
\label{Theorem_Quantization_VMO_Fock}
Let $f \in \textup{VMO}(\mathbb{C}^n)\cap L^{\infty}(\mathbb{C}^{n})$, then $\lim_{t \rightarrow 0} \| H_f^{(t)}\|_t=0$ and, in particular, we have 
\begin{equation*}
\lim_{t \rightarrow 0} \|T_f^{(t)} T_g^{(t)}-T_{fg}^{(t)}\|_t=0
\end{equation*}
for all $g \in L^{\infty}(\mathbb{C}^n)$ and all $g \in \textup{UC}(\mathbb{C}^n)$. 
\end{theorem}
\begin{proof}
It is sufficient to show that $\lim_{t \rightarrow 0} \| H_f^{(t)}\|_t=0$. In fact, this follows from Theorem \ref{Theorem_VMO_vanishing_BMO} and 
Corollary \ref{main_theorem_section_2}.
\end{proof}
%%%%%%%%%%%%%%%%%%%%%%%%%%%%%%%%%%%%%%%%%%%%%%%%%%%%%%%%%%%%%%%%%%%%%%%%%%%%%%%%%%%%%%%%%%%%%
\section{Fock quantization algebras}
\label{Section_8}
%%%%%%%%%%%%%%%%%%%%%%%%%%%%%%%%%%%%%%%%%%%%%%%%%%%%%%%%%%%%%%%%%%%%%%%%%%%%%%%%%%%%%%%%%%%%%
\noindent
We now consider the direct integral $L:= \int^{\oplus}_{\mathbb{R}_+} L_{t}^2$ of $L^2_t$-spaces. In the following we write decomposable operators $X$ on $L$ in the form 
$$ X = \oplus_{t > 0} X^{(t)} \hspace{3ex} \mbox{\it  for} \hspace{3ex} X^{(t)} \in \mathcal{L}(L^2_{t})\hspace{3ex} \mbox{\it with} \hspace{3ex} \|X\|:= \sup_{t >0} \| X^{(t)} \| < \infty.$$ 
The algebra of such operators is denoted by $\textup{Op}(L)$.  The set  
$$\mathcal{I} := \Big{\{} X \in \textup{Op}(L): \lim_{t \rightarrow 0} \|X^{(t)}\|_{t} = 0\Big{\}}$$ 
is a closed two-sided ideal in $\textup{Op}(L)$. We also consider the direct integral of the Fock spaces
$$H:= \int^{\oplus}_{\mathbb{R}_+} H^2_{t}.$$ 
\begin{remark}
{\rm 
The Toeplitz and Hankel operators $T_{f} = \oplus_{t > 0} T_{f}^{(t)}$ and $H_{f} = \oplus_{t > 0} H_{f}^{(t)} $ are in $\textup{Op}(L)$ for $f \in L^{\infty}$.  The Toeplitz operators are also in $\textup{Op}(H)$.}
\end{remark}
Noting that 
\begin{equation}\label{GL_semi_commutator_TO_section_8}
T_{f}^{(t)} T_{g}^{(t)} - T_{fg}^{(t)} = -\big{(}H_{\bar f}^{(t)}\big{)}^{\ast} H_{g}^{(t)}
\end{equation}
for all $t > 0$, we consider the set
$$
\mathcal{A}:= \Big{\{} f \in L^{\infty}(\mathbb{C}^n)\:  : \: \mbox{$T_{f} T_{g} - T_{fg}$ {\it and}  $T_{g} T_{f} - T_{fg} \in \mathcal{I}$ {\it for all} $g \in L^{\infty}(\mathbb{C}^n)$} \Big{\}}.
$$
\begin{proposition}\label{Proposition_A_closed_subalgebra}
$\mathcal{A}$ is a closed, conjugate-closed subalgebra of $L^{\infty}(\mathbb{C}^n)$ and we have 
\begin{equation}\label{GL_second_representation_of_A}
\mathcal{A}= \Big{\{} f \in L^{\infty}(\mathbb{C}^n) \: : \: H_f, H_{\overline{f}} \in \mathcal{I} \Big{\}}. 
\end{equation}
\end{proposition}
\begin{proof}
For $\{f_{n}\} \subset \mathcal{A}$, with $f_{n} \rightarrow f \in L^{\infty}(\mathbb{C}^n)$, it is easy to check that $f \in \mathcal{A}$.  For $f, h \in \mathcal{A}$ and $g \in L^{\infty}(\mathbb{C}^n)$, we have
\begin{equation}
T_{fh}T_{g} - T_{fhg}  \in T_{f}T_{h}T_{g} - T_{fhg} + \mathcal{I} \subset T_{f}T_{hg} - T_{f(hg)} +  \mathcal{I} \subset \mathcal{I}. 
\end{equation}
Hence $\mathcal{A}$ is an algebra, which clearly is invariant under complex conjugation of its elements. The equality of sets in (\ref{GL_second_representation_of_A}) follows from (\ref{GL_semi_commutator_TO_section_8}). 
\end{proof}
Note that the definition of the algebra $\mathcal{A}$ and its characterization in (\ref{GL_second_representation_of_A}) is given in terms of operator conditions. We now give an equivalent description which only involves 
a condition on the level of functions. Consider the set 
\begin{equation*}
\mathcal{B}:= \Big{\{} f \in L^{\infty} (\mathbb{C}^n) \: : \: \lim_{t \rightarrow 0}  \| f\|_{\textup{BMO}_{\ast}^t}=0 \Big{\}}. 
\end{equation*}
\begin{proposition}\label{algebra_A_equals_algebra_B}
$\mathcal{A}= \mathcal{B}$. In particular, $\mathcal{B}$ is a closed, conjugate-closed subalgebra of $L^{\infty}$. 
\end{proposition}
\begin{remark}
{\rm One may check directly from the definition of $\mathcal{B}$, that it is a  a closed, conjugate-closed sub-algebra of $L^{\infty}(\mathbb{C}^n)$. }
\end{remark} 
The next lemma serves as a preparation for the proof of Proposition \ref{algebra_A_equals_algebra_B}. It generalizes  (with the same proof) Lemma 2.5 in \cite{B} to the family of weighted spaces 
$L^2_t$ and $H^2_t$, $t>0$. 
\begin{lemma}\label{Lemma_section_8_BMO_versus_Hankel_norm}
Let $t>0$ and $f \in L^{\infty}(\mathbb{C}^n)$. Then the mean oscillation of $f$ can be represented in the form 
\begin{equation}\label{representation_mean_oscillation_chapter_8}
\textup{MO}^{(t)}(f)(w)=\big{\|} (I-P^{(t)})(f \circ \tau_w)\big{\|}_t^2+\big{\|} P^{(t)} [f \circ \tau_w] - \tilde{f}^{(t)}(w) \big{\|}_t^2, 
\end{equation}
where $\tau_w(z):=w-z$ for $z,w \in \mathbb{C}^n$ denotes the translation by $w$. Moreover, 
\begin{equation}\label{Abschaetzung_BMO_by_Hankel_operators}
\|f\|_{\textup{BMO}^t_{\ast}} \leq \sqrt{ \|H_f^{(t)}\|_t^2+ \|H^{(t)}_{\overline{f}}\|_t^2}. 
\end{equation}
\end{lemma}
\begin{proof}
Let $w \in \mathbb{C}^n$ be fixed and note that $\langle P^{(t)}(f \circ \tau_w), \tilde{f}^{(t)}(w) \rangle_t= |\tilde{f}^{(t)}(w)|^2$. Therefore 
\begin{equation*}
\big{\|} P^{(t)} [f \circ \tau_w] - \tilde{f}^{(t)}(w) \big{\|}_t^2
= \big{\|} P^{(t)} [f \circ \tau_w] \big{\|}_t^2-|\tilde{f}^{(t)}(w)|^2. 
\end{equation*}
\par 
Using this relation and (\ref{Property_and_defn_of_MO}) we can express the mean oscillation in the form  (\ref{representation_mean_oscillation_chapter_8}):
\begin{align*}
\textup{MO}^{(t)}(f)(w)
&= \|f\circ \tau_w\|_t^2- |\tilde{f}^{(t)}(w)|^2\\
&= \big{\|} (I-P^{(t)})[f \circ \tau_w]\big{\|}_t^2+ \big{\|}P^{(t)} [f \circ \tau_w]\big{\|}_t^2-|\tilde{f}^{(t)}(w)|^2\\
&= \big{\|} (I-P^{(t)})[f \circ \tau_w]\big{\|}_t^2+\big{\|} P^{(t)} [f \circ \tau_w] - \tilde{f}^{(t)}(w) \big{\|}_t^2.
\end{align*}
\par 
Now we can prove the estimate (\ref{Abschaetzung_BMO_by_Hankel_operators}). Note that  (\ref{heat_transform}) implies: 
\begin{align*}
\tilde{f}^{(t)}(w)
= \big{\langle} f \circ \tau_w, 1 \big{\rangle}_t
&= 
\overline{
\big{\langle} \overline{f}\circ \tau_w, K_t(\cdot, 0) \big{\rangle}_t}=\\
&=
 \overline{
 P^{(t)}( \overline{f} \circ \tau_w)(0)}
= P^{(t)} \big{(} \overline{P^{(t)}(\overline{f} \circ \tau_w)} \big{)}. 
\end{align*}
It follows: 
\begin{align*}
\big{\|} P^{(t)} [f \circ \tau_w] - \tilde{f}^{(t)}(w) \big{\|}_t^2
&=\big{\|} P^{(t)} [f \circ \tau_w] - P^{(t)} \big{(} \overline{P^{(t)}(\overline{f} \circ \tau_w)}\big{)} \big{\|}_t^2\\
& \leq \big{\|} f \circ \tau_w - \overline{P^{(t)}(\overline{f} \circ \tau_w)} \big{\|}_t^2=\big{\|} (I-P^{(t)}) ( \overline{f} \circ \tau_w) \big{\|}_t^2. 
\end{align*}
Together with (\ref{representation_mean_oscillation_chapter_8}) one obtains: 
\begin{align*}
\textup{MO}^{(t)}(f)(w)
&\leq \big{\|} (I-P^{(t)})[f \circ \tau_w]\big{\|}_t^2+ \big{\|} (I-P^{(t)}) [\overline{f} \circ \tau_w]\big{\|}_t^2\\
&=\big{\|} H_f^{(t)}k_w^{t}\big{\|}_t^2+  \big{\|} H^{(t)}_{\overline{f}}k_w^t\big{\|}_t^2\leq \| H_f^{(t)}\|^2_t + \|H_{\overline{f}}^{(t)} \|_t^2. 
\end{align*}
Now, (\ref{Abschaetzung_BMO_by_Hankel_operators}) follows from the definition of the $\textup{BMO}_{\ast}^t$-semi-norm. 
\end{proof}
\begin{proof}({\it Proposition \ref{algebra_A_equals_algebra_B}}). 
The inclusion $\mathcal{B} \subset \mathcal{A}$ directly follows from Corollary \ref{main_theorem_section_2} whereas $\mathcal{A} \subset \mathcal{B}$ is 
a consequence of (\ref{GL_second_representation_of_A}) in Proposition \ref{Proposition_A_closed_subalgebra} and (\ref{Abschaetzung_BMO_by_Hankel_operators}) 
in Lemma \ref{Lemma_section_8_BMO_versus_Hankel_norm}. 
\end{proof}
Note that, by Theorem \ref{theorem_VMO_algebra} the space $\textup{VMO}(\mathbb{C}^n)\cap L^{\infty}(\mathbb{C}^n)$ is a closed, conjugate-closed subalgebra of $L^{\infty}(\mathbb{C}^n)$. From Theorem \ref{Theorem_VMO_vanishing_BMO} we know that 
\begin{equation}\label{Inclusion_VMO_bounded_B_A}
\textup{VMO}(\mathbb{C}^n) \cap L^{\infty}(\mathbb{C}^n) \subset \mathcal{B} = \mathcal{A}. 
\end{equation}
%{\bf Conjecture:} All three algebras above coincide, i.e.  $\textup{VMO}(\mathbb{C}^n) \cap L^{\infty} (\mathbb{C}^n)= \mathcal{B} = \mathcal{A}$. 
%\noindent
%\vspace{1ex}\\
\begin{lemma}\label{Lemma_estimate_Var_by_MO}
There is $C_n>0$ only depending on the dimension $n$ such that for all $a \in \mathbb{C}^n$: 
\begin{equation*}
\| f\|_{\textup{BMO}^t_{\ast}}^2\geq 
\textup{MO}^t(f)(a)\geq C_n   \textup{Var}_{B(a,\sqrt{t})} (f). 
\end{equation*}
In particular, the following inclusion holds: $\mathcal{B} \subset \textup{VMO}(\mathbb{C}^n) \cap L^{\infty}(\mathbb{C}^n).$
\end{lemma}
\begin{proof}
It is easy to check (see \cite{BBCZ}) that 
\begin{equation*}
\textup{MO}^t(f)(a)= \frac{1}{2} \int_{\mathbb{C}^n} \int_{\mathbb{C}^n} |f(u)-f(w)|^2 d\mu_{t,a}(u) d\mu_{t,a}(w), 
\end{equation*}
where for $a,u \in \mathbb{C}^n$ we define: 
\begin{equation*}
d\mu_{t,a}(u):= |k_a^t(u)|^2 d\mu_t(u)= \frac{1}{(4\pi t)^n} e^{- \frac{1}{4t} |a-u|^2}dv(u). 
\end{equation*}
Hence, by applying Lemma \ref{lemma_representation_ Var}, we can estimate the mean oscillation from below as follows:
\begin{align*}
\textup{MO}^t(f)(a)
& \geq \frac{1}{2} \frac{1}{(4 \pi t)^{2n}} \int_{B(a, \sqrt{t})} \int_{B(a,\sqrt{t})} |f(u)-f(w)|^2 e^{- \frac{|a-u|^2+|a-w|^2}{4t}} dv(u) dv(w)\\
& \geq  \frac{e^{-\frac{1}{2}}}{(4\pi t)^{2n}}\cdot \frac{1}{2}\int_{B(a,\sqrt{t})} \int_{B(a,\sqrt{t})} |f(u)-f(w)|^2 dv(u) dv(w)\\
&= e^{- \frac{1}{2}} \frac{|B(a,\sqrt{t})|^2}{(4\pi t)^{2n}} \textup{Var}_{B(a,\sqrt{t})} (f). 
\end{align*}
Note that there is a constant $C_n$ only depending on the complex dimension $n$ such that 
\begin{equation*}
e^{- \frac{1}{2}} \frac{|B(a,\sqrt{t})|^2}
{(4\pi t)^{2n}}=C_n.  
\end{equation*}
From this, the statement follows. 
\end{proof}
Finally, (\ref{Inclusion_VMO_bounded_B_A}) together with Lemma \ref{Lemma_estimate_Var_by_MO} shows:
\begin{theorem}\label{Main_theorem_Section_8}
$\mathcal{A}= \mathcal{B}= \textup{VMO}(\mathbb{C}^n) \cap L^{\infty}(\mathbb{C}^n)$. 
\end{theorem}
Theorem \ref{Main_theorem_Section_8} indicates that in case of the ideal $\mathcal{I}$ the algebra $\textup{VMO}(\mathbb{C}^n) \cap L^{\infty}(\mathbb{C}^n)$ plays a similar role as 
$\textup{VO}(\mathbb{C}^n)$ in case of compact operators (cf. Proposition \ref{compact_semi_commutators}). 
%%%%%%%%%%%%%%%%%%%%%%%%%%%%%%%%%%%%%%%%%%%%%%%%%%%%%%%%%%%%%%%%%%%%%%%%%%%%%%%%%%
\section{On the limit of the norm of Toeplitz operators}
\label{new_Section_6}
%%%%%%%%%%%%%%%%%%%%%%%%%%%%%%%%%%%%%%%%%%%%%%%%%%%%%%%%%%%%%%%%%%%%%%%%%%%%%%%%%%%
For each function $f$ in $\textup{BUC}(\C^n)$, \cite[Theorem 1]{C} shows the following identity
\begin{equation}\label{norm_limit}
\lim\limits_{t \to 0} \|T_{f}^{(t)}\|_t = \|f\|_{\infty}.
\end{equation}

Here, we extend this result to operator symbols $f \in L^{\infty}(\C^n)$ by showing that $\tilde{f}^{(t)}$ converges pointwise almost everywhere to $f$. The result will then be a consequence of Remark \ref{heat_remark}.

Recall that the heat transform $\big{(}\tilde{f}^{(t)}\big{)}_{t > 0}$ has the semi-group property, i.e.~for $s,t>0$ and $f \in L^{\infty}(\mathbb{C}^n)$ it holds
\begin{equation*}
\big{\{} \tilde{f}^{(t)} \widetilde{\big{\}}}^{(s)}= \widetilde{f}^{(s+t)}. 
\end{equation*}
Moreover, the assignment $f \mapsto \tilde{f}^{(t)}$ is a contraction on $L^{\infty}(\C^n)$, i.e.~$\|\tilde{f}^{(t)}\|_{\infty} \leq \|f\|_{\infty}$ for every $t > 0$. Letting $s > t > 0$ and combining these properties shows
\begin{equation*}
\|\tilde{f}^{(s)}\|_{\infty} = \big{\|} \{ \tilde{f}^{(t)} \widetilde{\}}^{(s-t)} \big{\|}_{\infty} \leq \| \tilde{f}^{(t)}\|_{\infty}. 
\end{equation*}
Therefore $t \mapsto \|\tilde{f}^{(t)} \|_{\infty}$ is monotone decreasing and $\lim\limits_{t \rightarrow 0} \| \tilde{f}^{(t)}\|_{\infty}$ exists for all $f \in L^{\infty}(\mathbb{C}^n)$.

As a further preparation, we show that the heat transform is bounded by the Hardy-Littlewood maximal function for any locally integrable $f$. The Hardy-Littlewood maximal function $f^*$ is defined by
\[f^*(w) := \sup\limits_{r > 0} \frac{1}{|B(0,r)|} \int_{B(0,r)} |f(w-z)| \, dv(z),\]
where $|B(0,r)|$ denotes the (Lebesgue) volume of the ball with radius $r$ as usual.

\begin{lemma} \label{Hardy_Littlewood}
For every $t > 0$, $w \in \C^n$ and any locally integrable function $f$ we have $|\tilde{f}^{(t)}(w)| \leq Cf^*(w)$, where $C$ is some constant that only depends on the dimension $n$.
\end{lemma}

\begin{proof}
Recall that the volume of a (real) $2n$-dimensional ball of radius $r$ is given by $\frac{\pi^n}{n!}r^{2n}$. The result now follows by the following computation:
\begin{align*}
|\tilde{f}^{(t)}(w)| &\leq \frac{1}{(4\pi t)^n} \int_{\C^n} |f(w-z)| e^{-\frac{|z|^2}{4t}} \, dv(z)\\
&\leq \frac{1}{(4\pi t)^n} \int_{B(0,\sqrt{4t})} |f(w-z)| \, dv(z) + \frac{1}{(4\pi t)^n} \int_{B(0,\sqrt{8t}) \setminus B(0,\sqrt{4t})} |f(w-z)|e^{-1} \, dv(z)\\
&\quad + \ldots\\
&\leq \sum\limits_{k = 1}^{\infty} \frac{k^n}{n!} \frac{1}{|B(0,\sqrt{4kt})|} \int_{B(0,\sqrt{4kt})} |f(w-z)| e^{-(k-1)} \, dv(z)\\
&\leq \frac{1}{n!}\sum\limits_{k = 1}^{\infty} k^ne^{-(k-1)} f^*(w)\\
&=: Cf^*(w).\qedhere
\end{align*}
\end{proof}

Now we are ready to prove the theorem announced at the beginning of this section.

\begin{theorem} \label{heat_convergence_L_infinity}
Let $f \in L^{\infty}(\C^n)$. Then $\lim\limits_{t \to 0} \tilde{f}^{(t)}(w) = f(w)$ for almost every $w \in \C^n$. In particular,
\[\lim\limits_{t \to 0} \|\tilde{f}^{(t)}\|_{\infty} = \lim\limits_{t \to 0} \|T_f^{(t)}\|_t = \|f\|_{\infty}.\]
\end{theorem}

\begin{proof}
Let $f \in L^{\infty}(\mathbb{C}^n)$, $K \subset \C^n$ compact and $\delta, \epsilon > 0$. Choose $r$ sufficiently large such that $K \subset B(0,r)$ and
\begin{equation} \label{gaussian_balls_size}
\int_{\C^n \setminus B(w,r)} \, d\mu_t(z) < \frac{\epsilon}{\|f\|_{\infty}}
\end{equation}
for all $t \in (0,1)$ and $w \in K$. To see that this is possible observe that for sufficiently large $|z|$ the weight $\frac{1}{(4\pi t)^n} e^{-\frac{|z|^2}{4t}}$ is decreasing as $t \to 0$. We may thus choose a radius $r_0$ such that \eqref{gaussian_balls_size} holds for some $w_0 \in K$. Adding the diameter of $K$ to $r_0$ yields a sufficiently large radius $r$.

Cutting $f$ at that radius, i.e.~setting $g_0 := f \cdot \chi_{B(0,r)} \in L^1(\C^n,dv)$, we get
\[|\tilde{f}^{(t)}(w) - \tilde{g_0}^{(t)}(w)| = \int_{\C^n \setminus B(w,r)} |f(w-z)| \, d\mu_t(z) < \epsilon\]
for all $w \in K$ and $t \in (0,1)$. Now choose a continuous function $g_1$ of compact support such that $\|g_0-g_1\|_{L^1} < \delta$. Then, by the inequality above and Proposition \ref{Lemma_convergence_heat_transform_UC_functions},
\begin{align*}
|\tilde{f}^{(t)}(w) - f(w)| &\leq |\tilde{f}^{(t)}(w) - \tilde{g_0}^{(t)}(w)| + |\tilde{g_0}^{(t)}(w) - \tilde{g_1}^{(t)}(w)| + |\tilde{g_1}^{(t)}(w) - g_1(w)|\\
&\quad + |g_1(w) - g_0(w)| + |g_0(w) - f(w)|\\
&< 2\epsilon + |\tilde{g_0}^{(t)}(w) - \tilde{g_1}^{(t)}(w)| + |g_1(w) - g_0(w)|
\end{align*}
for all $w \in K$ and sufficiently small $t$. To obtain the assertion, we need to show that
\[\{w \in \C^n : \limsup_{t \to 0} |\tilde{f}^{(t)}(w) - f(w)| > 4\epsilon\}\]
is a null set for all $\epsilon > 0$.

By Lemma \ref{Hardy_Littlewood}, we have
\[|\tilde{g_0}^{(t)}(w) - \tilde{g_1}^{(t)}(w)| \leq C(g_0-g_1)^*(w)\]
and as is well-known, the Hardy-Littlewood maximal function satisfies the weak (1,1)-inequality, i.e.~there exists a constant $C_1$ depending only on the dimension $n$ such that
\[|\{w \in \C^n : g^*(w) > \epsilon\}| \leq \frac{C_1}{\epsilon}\|g\|_{L^1}\]
for all $g \in L^1(\C^n,dv)$ (see \cite[Theorem 1]{SS}). Applying this to $g_0-g_1$, we obtain
\begin{align*}
|\{w \in \C^n : |\tilde{g_0}^{(t)}(w) - \tilde{g_1}^{(t)}(w)| > \epsilon\}| &\leq |\{w \in \C^n : (g_0-g_1)^*(w) > \frac{\epsilon}{C}\}| \leq \frac{CC_1}{\epsilon}\|g_0-g_1\|_{L^1}\\
&< \frac{CC_1}{\epsilon}\delta.
\end{align*}
Moreover,
\[|\{w \in K : |g_1(w) - g_0(w)| > \epsilon\}| \leq \frac{1}{\epsilon}\|g_0-g_1\|_{L^1} < \frac{\delta}{\epsilon}\]
by Markov's inequality. Thus
\[|\{w \in K : \limsup_{t \to 0} |\tilde{f}^{(t)}(w) - f(w)| > 4\epsilon\}| < \frac{CC_1+1}{\epsilon}\delta.\]
Since $\delta$ was arbitary, we get
\[|\{w \in K : \limsup_{t \to 0} |\tilde{f}^{(t)}(w) - f(w)| > 4\epsilon\}| = 0\]
for all $\epsilon > 0$ and all compact sets $K \subset \C^n$. Clearly, by taking a countable covering, this remains true if we replace $K$ by $\C^n$ in the above formula. Therefore $\tilde{f}^{(t)}(w)$ 
converges to $f(w)$ for almost every $w \in \C^n$.

To prove the second assertion choose for every $\epsilon > 0$ a bounded set $A_{\epsilon} \subset \C^n$ with $|A_{\epsilon}| > 0$ such that $|f(w)| \geq \|f\|_{\infty} - \epsilon$ for all $w \in A_{\epsilon}$. 
By Egorov's theorem, we can additionally assume $\tilde{f}^{(t)}(w) \to f(w)$ uniformly for all $w \in A_{\epsilon}$. It follows
\[\|f\|_{\infty} \geq \lim\limits_{t \to 0} \|\tilde{f}^{(t)}\|_{\infty} \geq \lim\limits_{t \to 0} \|\tilde{f}^{(t)}|_{A_{\epsilon}}\|_{\infty} \geq \|f\|_{\infty} - \epsilon\]
for all $\epsilon > 0$. Using Remark \ref{heat_remark} and the obvious inequality $\|T_f^{(t)}\|_t \leq \|f\|_{\infty}$, we conclude
\[\lim\limits_{t \to 0} \|\tilde{f}^{(t)}\|_{\infty} = \lim\limits_{t \to 0} \|T_f^{(t)}\|_t = \|f\|_{\infty}.\qedhere\]
\end{proof}

%%%%%%%%%%%%%%%%%%%%%%%%%%%%%%%%%%%%%%%%%%%%%%%%%%%%%%%%%%%%%%%%%%%%%%%%%%
\section{Examples}
\label{examples}
%%%%%%%%%%%%%%%%%%%%%%%%%%%%%%%%%%%%%%%%%%%%%%%%%%%%%%%%%%%%%%%%%%%%%%%%%%%%%

In this section we provide three explicit examples. The first two are counterexamples to the statements of Theorem \ref{Theorem_behaviour_semi_commutator_uniformly_bounded_function} and Theorem \ref{Theorem_Quantization_VMO_Fock}. Clearly, these functions cannot be uniformly continuous or have vanishing mean oscillation. The third example shows that the semi-commutator of two unbounded Toeplitz operators can be bounded and that its norm can still tend to $0$ as $t \to 0$.

\noindent
{\bf (A):} Direct calculation shows that a natural orthonormal basis for $H^{2}(\mathbb{C}, d\mu_{t})$ consists of the functions
\begin{equation*}
e^{(t)}_{k} = \frac{z^k}{\sqrt{(4t)^{k} k!}}, \hspace{3ex}  \mbox{\it where} \hspace{3ex}  k \in \mathbb{Z}_+ = \{ 0, 1, 2, 3, \ldots\}. 
\end{equation*}
Moreover, for $f(z) = \exp(i |z|^{2})$ and $g(z) = \exp(-i |z|^{2})$, we see that $T^{(t)}_{f} , T^{(t)}_{g}$ are diagonal in the orthonormal basis $\{e^{(t)}_{k}: k \in \mathbb{Z}_+\}$ with eigenvalues
\begin{equation*}
s_{k}\big{(} \exp(i|z|^{2})\big{)} = (1- 4ti )^{-(k + 1)} \hspace{2ex}\mbox{\it and} \hspace{2ex}  s_{k}\big{(}\exp(-i|z|^{2})\big{)} = (1 + 4ti )^{-(k + 1)}.
\end{equation*}
Since $f(z)g(z)\equiv 1$,
\begin{equation*}
\|T^{(t)}_{f} T^{(t)}_{g} - T^{(t)}_{fg}\|_{t}= \sup \big{\{}1 - (1 + 16t^{2})^{-(k + 1)} : k  \in \mathbb{Z}_+\big{\}}= 1
\end{equation*}
for all $t > 0$. We also observe that $[T^{(t)}_{f}, T^{(t)}_{g}] = 0 = \{f, g\} $.
\vspace{1ex}\\ 
{\bf (B)}: Here is another counterexample to the quantization result in Theorem \ref{Theorem_behaviour_semi_commutator_uniformly_bounded_function} and Theorem \ref{Theorem_Quantization_VMO_Fock}. Different from the example in (A) we choose a symbol which has high oscillation inside the domain (in a zero-neighbourhood). Such an effect was already observed in \cite{BC00}, however, 
the present example is even simpler. Let $n=1$ and consider the following symbol: 
\begin{equation*}
f(z):=
\begin{cases}
0, & \mbox{\it if $z=0$} \\
1, & \mbox{\it if $2^j\leq |z| <2^{j+1}$ and $j \in \mathbb{Z}$ is even}\\  
-1, & \mbox{\it if $2^j \leq |z| < 2^{j+1}$ and $j \in \mathbb{Z}$ is odd}. 
\end{cases}
\end{equation*}
We have $-f(z)= f(\frac{z}{2})$ for all $z \in \mathbb{C}$. Let $U_t: H_t^2 \rightarrow H^2_{\frac{1}{4}}$ be the family of isometries defined in 
(\ref{operator_U_t}).  As was observed in (\ref{Transformation_Toeplitz_operator}) the Toeplitz operator $T_{f(\cdot 2 \sqrt{t})}^{(\frac{1}{4})}$ transforms under conjugation by $U_t$ as follows: 
\begin{equation*}
T_f^{(t)}= U_t^*T_{f(\cdot 2 \sqrt{t})}^{(\frac{1}{4})} U_t. 
\end{equation*}
%Let $U_\frac{t}{4}:H^2_t\rightarrow H^2_1=:H^2$, $t>0$ be the family of unitary operators defined by a dilation $U_tg:= g(\cdot \sqrt{t})$. 
%One easily checks (cf. \cite{BCI}) that $U_t^*T_{f(\cdot \sqrt{t})}^{(1)}U_t= T_f^{(t)}$. 
Choose a sequence $t_{\ell}:= 4^{-{\ell}-1}$ where $\ell \in \mathbb{N}$. Then we have 
\begin{equation*}
T_f^{(t_{\ell})} T_f^{(t_{\ell})}-T_{f^2}^{(t_{\ell})} 
%= T_f^{(t_{\ell})} T_f^{(t_{\ell})} -I
= U_{t_{\ell}}^* \Big{[} T_{f(\cdot 2^{-\ell})}^{(\frac{1}{4})} T_{f(\cdot 2^{-\ell})}^{(\frac{1}{4})} -I \Big{]} U_{t_{\ell}}. 
\end{equation*}
Because of $f(\cdot 2^{-{\ell}})=(-1)^{\ell} f$ we obtain: 
\begin{equation*}
\big{\|} T_f^{(t_{\ell})} T_f^{(t_{\ell})}-T_{f^2}^{(t_{\ell})} \big{\|}_{t_{\ell}}=\big{\|} T_f^{(\frac{1}{4})} T_f^{(\frac{1}{4})} -T_{f^2}^{(\frac{1}{4})} \big{\|}_{\frac{1}{4}}
=\big{\|} T_f^{(\frac{1}{4})} T_f^{(\frac{1}{4})} -I \big{\|}_{\frac{1}{4}}=(*). 
\end{equation*}
One can check (e.g. using the fact that $T_f^{(\frac{1}{4})}$ is a diagonal operator) that $(*)$ is non-zero. Since ($*$) also does not depend on $\ell$ we cannot have $\lim_{t\rightarrow 0} \|T_f^{(t)} T_f^{(t)} -T_{f^2}^{(t)}\|_t=0$. 
\vspace{1ex}\\
{\bf (C):} We consider the annihilation and creation operators $T_{\bar z}^{(t)}, T_{z}^{(t)} $ on the one-particle bosonic Fock space $H^{2}(\mathbb{C}, d\mu_{t})$ for all $t > 0$. 
Standard calculation shows for $k  \in \mathbb{Z}_+$: 
\begin{align*}
T_{z}^{(t)} e_{k}^{(t)} 
&= \{(4t)(k + 1)\}^{1/2} e_{k + 1}^{(t)}, \\
T_{\bar z}^{(t)} e_{k + 1}^{(t)} 
&= \{(4t)(k + 1)\}^{1/2} e_{k}^{(t)} \hspace{3ex} \mbox{\it and} \hspace{3ex} T_{\bar z}^{(t)} e_{0}^{(t)} = 0.
\end{align*}
It follows for $k \in \mathbb{Z}_+$ that 
\begin{equation*}
[T_{\bar z}^{(t)}, T_{z}^{(t)}] e_{k}^{(t)} = 4t e_{k}^{(t)}.%, \hspace{4ex} k \in \mathbb{Z}_+
\end{equation*}
Noting that $T_{\bar z}^{(t)}T_{z}^{(t)} = T_{{\bar z} z}^{(t)}$ yields
\begin{equation*}
\|T_{ z}^{(t)}T_{\bar z}^{(t)} - T_{z{\bar z}}^{(t)}\|_{t} = 4t \rightarrow 0 \hspace{3ex} \mbox{\it as} \hspace{3ex} t \rightarrow 0. 
\end{equation*}
{\bf Remark.}  
Direct computational checks of Theorem \ref{Theorem_behaviour_semi_commutator_uniformly_bounded_function} in the diagonal case bring us to (or even over) the edge of what is 
possible using Stirling's approximation. For an example, consider the estimation of
\begin{equation*}
\| T_{|z|}^{(t)} T_{|z|}^{(t)} - T_{|z|^{2}}^{(t)} \|_{t}.
\end{equation*}
\noindent
{\bf Acknowledgement:} We thank Jingbo Xia for his useful conversation and comments. In particular, he pointed out the simple proof of Corollary  \ref{main_theorem_section_2} which 
replaced our previous (and slightly weaker) estimate on the norm of the Hankel operator $H_f^{(t)}$. 
%%%%%%%%%%%%%%%%%%%%%%%%%%%%%%%%%%%%%%%%%%%%%%%%%%%%%%%%%%%%%%%%%%%%%%%%%%%%%%%%%%%%
%%%%%%%%%%%%%%%%%%%%%%%%%%%%%%%%%%%%%%%%%%%%%%%%%%%%%%%%%%%%%%%%%%%%%%%%%%%%%%%%%%%%%

\end{document}